\documentclass[20pt]{article}

\input epsf
\usepackage{amssymb}
\usepackage{amsfonts}
\usepackage{amsmath}
\usepackage{mathrsfs}
\usepackage{amsthm}

\numberwithin{equation}{section}
\newtheorem{theorem}{Theorem}[section]
\newtheorem{lemma}{Lemma}[section]

\newtheorem{corollary}{Corollary}[section]

\newtheorem{remark}{Remark}[section]

\begin{document}

 \baselineskip 20pt

\textwidth=146truemm
\textheight=207truemm

\begin{center}

 {\bf \Large A spline interpretation of Eulerian numbers \footnote{The Project Supported by The National Natural
Science Foundation of China (NO.60373093,  NO.60533060, NO.10871196
and NO.10726068.) }}
\end{center}

\begin{center}
{Renhong Wang \footnote{Email: renhong@dlut.edu.cn},    Yan Xu
\footnote{Email: yanxudlut@yahoo.cn}

 {\small Institute of
Mathematical Sciences, Dalian University of Technology,\newline
 Dalian, 116024, China} }

 {Zhiqiang Xu\footnote{Email: xuzq@lsec.cc.ac.cn.}}

{\small LSEC, Inst. Comp. Math., Academy of Mathematics and System
Sciences, \newline Chinese
 Academy of Sciences, Beijing, 100080 China }
\end{center}

\begin{abstract}
In this paper, we explore the interrelationship between Eulerian
numbers and B splines. Specifically, using B splines, we give the
explicit formulas of the refined Eulerian numbers, and  descents
polynomials. Moreover,   we prove that the coefficients of descent
polynomials $D_d^n(t)$ are log-concave. This paper also provides a
new approach to study Eulerian numbers and descent polynomials.

\end{abstract}

 \vspace{0.5cm}

\section{Introduction}
Denote by $C^d$ the $d$-dimensional unit cube, i.e.,
           $$C^{d}\,\,=\,\,\{x\in \mathbb{R}^{d} : 0\leq x_i\leq 1,\,\, i=1,\ldots,d \}.$$
For a real number  $\lambda$, the {\it dilation} of $C^{d}$ by
$\lambda$ is the set
 $$\lambda C^{d}\,\,=\,\,\{\lambda\cdot x: x \in C^{d}\}.$$
Throughout this paper, all the volumes and mixed volumes  will be
normalized so that the volume of $C^{d}$ is given by $V(C^{d})=d!$,
where $V(\cdot)$ denotes the volume function (cf. \cite{E}). Denoted
by $T^{d}_k$ and $X^d_{n,k}$,  the $k$th slices of $C^{d}$ and $ n
\cdot C^{d}$ respectively:
\begin{eqnarray*}
T^{d}_k&=&\{x \in C^{d} : k-1\leq \sum_{i=1}^{d}x_{i}\leq k\},\\
X^d_{n,k}&=&\{x\in nC^{d}: (k-1)n+1\leq\sum_{i=1}^dx_{i}\leq kn+1\},
\end{eqnarray*}
where $n$ is a positive integer.

The Eulerian number $A_{d,k}$ is the number of permutations in the
symmetric group $S_d$ that have exactly $k-1$ descents. It seems
that Laplace  first expressed the volume of $T_k^d$ in terms of
$A_{d,k}$ as follows \cite{l}:
\begin{equation}\label{eq:1.1}
A_{d,k}\,\,=\,\, V(T^{d}_{k}).
\end{equation}

The refinement of the Eulerian number $\mathbf{A}_{d,k,j}$ is the
number of permutations in the symmetric group $S_d$ with $k$
descents and ending with the element $j$. A result due to Ehrenborg,
Readdy and  Steingr\'{\i}msson \cite{E} generalized Laplace's result
to express the mixed volumes of two adjacent slices from the unit
cube in terms of the refinement of the Eulerian numbers
$\mathbf{A}_{d+1,k,d+1-j}$, i.e.,
\begin{equation}
\mathbf{A}_{d+1,k,d+1-j}=V(T^d_k , j; T^d_{k+1}, d-j).
\end{equation}
The number $V(T_k^d,j; T_{k+1}^d,d-j)$ is the $(d-j)$-th mixed
volume of the convex bodies $T_k^d$ and $T_{k+1}^d$ (cf.
\cite{stan}). For $\lambda, \mu\geq 0$, define the Minkowski sum
 $$ \lambda T_k^d+\mu T_{k+1}^d=\{\lambda \alpha + \mu \beta :\alpha \in T_k^d,\beta \in T_{k+1}^d\}.$$
It was shown essentially by Minkowski (though he treated only the
case $d\leq 3$) that there are real numbers $V(T_{k}^d,
d-j;T_{k+1}^d,j)\geq 0$ satisfying
\begin{equation}
V(\lambda T_{k}^d+\mu T_{k+1}^d)=\sum _{j=0}^{d}\binom {d}{j}
V(T_{k}^d, d-j;T_{k+1}^d,j) \lambda^{d-j}\mu^j,
\end{equation}
for all $\lambda,\mu \geq 0$. The number $V(T_{k}^d,
d-j;T_{k+1}^d,j)$ is called the $j$-th mixed volume of $T_{k}^d$ and
$T_{k+1}^d$.

Descent polynomials, denoted by $D^n_d(t)$, are defined as
$$
\sum_{k=0}^d D(d,n,k)t^k,
$$
where $D(d,n,k)$  is the number of permutations in indexed
permutation $S^n_d$ with $k$ descents. The indexed permutation
$S^n_d$ of length $d$ and with indices in $\{0,1,\ldots,n-1\}$ is an
ordinary permutation in the symmetric group $S_d$ where each letter
has been assigned an integer between $0$ and $n-1$ (see
\cite{E,Stein}). Steingr\'{\i}msson \cite{Stein} first investigated
the relationship between $D(d,n,k)$ and the volume of the slice of
$nC^d$. To give a combinatorial interpretation for the mixed volumes
of two adjacent slices from the unit cube, Ehrenborg, Readdy and
Steingr\'{\i}msson \cite{E} found the refinement of the Eulerian
numbers are associated with $D(d,n,k)$:
\begin{equation}
D(d,n,k)=V(X^d_{n,k})=\sum_{j=0}^d{d\choose j}
\mathbf{A}_{d+1,k,d+1-j}(n-1)^j.
\end{equation}
Steingr\'{\i}msson \cite{Stein} gave the following recurrence
relations of $D(d,n,k)$,
\begin{equation}
 D(d,n,k)=(nk+1)D(d-1,n,k)+(n(d-k)+(n-1))D(d-1,n,k-1).
\end{equation}

 We now turn to the definition of B splines, and show the relation
 among Eulerian numbers,  descent polynomials and B splines.
{\it  B splines with order $d$}, denoted by $B_d(\cdot)$,   is
defined by the induction as follows
$$
B_1(x)\,\,=\,\,\begin{cases}
1 & \text{if} \>\> x\in[0,1),\\
0 & \text {otherwise},
\end{cases}
$$
and for $d\geq 2$
$$
B_d(x)\,\,=\,\, \int_0^1 B_{d-1}(x-t)dt.
$$
 A well known explicit formula for
$B_d(\cdot)$ is
\begin{equation}\label{eq:bspline}
B_d(x)=\frac{1}{(d-1)!}\sum_{i=0}^d\binom{d}{i}(-1)^i(x-i)_+^{d-1}.
\end{equation}
There is a recurrence relationship
\begin{equation}\label{eq:recbspline}
B_d(x)=\frac{x}{d-1}B_{d-1}(x)+\frac{d-x}{d-1}B_{d-1}(x-1).
\end{equation}
 From the viewpoint of the discrete
geometry, $B_d(x)$ equals the volume of slice of unit cubes (see
\cite{de boor}), i.e.,
\begin{equation}\label{eq:unitBspline}
B_d(x)=(d-1)!\cdot V(  \mathcal{H }\cap C^{d}),
\end{equation}
where $\mathcal{H}=\{y : y_1+\cdots +y_{d}=x\}$.

Based on (\ref{eq:unitBspline}), we can derive the relationship
between B splines and Eulerian numbers. To state conveniently,  we
use $[\lambda^j]f(\lambda)$ to denote the coefficient of $\lambda^j$
in $f(\lambda)$ for any given power series $f(\lambda)$. Then the
following theorem shows the relationship between  B splines and
Eulerian numbers.
\begin{theorem}\label{pr:1}
$$
  A_{d,k} \,=\, d!\cdot B_{d+1}(k);\leqno{\rm (i)}
$$
$$
   \mathbf{A}_{d+1,k,d-j+1}\, =\, d!\cdot [
\lambda^j]\left(
(\lambda+1)^dB_{d+1}(k+\frac{1}{\lambda+1})\right)/{\binom{d}{j}},\,\,
\lambda\geq 0;\leqno{\rm (ii)}
$$
$$
 D(d,n,k)\,=\, d!\cdot n^d\cdot B_{d+1}(k+\frac{1}{n}).\leqno{\rm (iii)}
$$
\end{theorem}

 We can also
present the explicit expressions for $\mathbf{A}_{d+1,k,d-j+1}$ and
$D(d,n,k)$:
\begin{corollary}\label{pr:2}
$$
\mathbf{A}_{d+1,k,d-j+1}=\sum_{i=0}^{k}\binom{d+1}{i}(-1)^i
(k-i)^j(k-i+1)^{d-j}; \leqno{\rm (i)}
$$
 $$
D(d,n,k)=\sum_{i=0}^{k}\binom{d+1}{i}(-1)^i(n(k-i)+1)^d.\leqno{\rm
(ii) }$$
\end{corollary}

In \cite{Stein},  Steingr\'{\i}msson proved that $D(d,n,k)$ for
$k=0,\ldots,d$ are unimodal. Using B splines, we have
\begin{corollary}\label{th:1}
For any $d$ and $n$, the sequence $D(d,n,k)$ for $k=0,\ldots,d$ is
log-concave.
\end{corollary}

The two-scale property of B splines also implies the two-scale
equations of $A_{d,k}$ and $D(d,n,k)$.

\begin{corollary}\label{th:2}
The Eulerian numbers $A_{d,k}$ and $D(d,n,k)$  satisfy the two-scale
equations
$$
A_{d,k}=\sum_{j=0}^{d+1} 2^{-d}\binom {d+1}{j}A_{d,{2k-j}},
$$
and
$$
 D(d,2n,k)=\sum_{j=0}^{d+1}\binom {d+1}{j}D(d,n,{2k-j}),
$$
respectively.
\end{corollary}
\begin{remark}
 In  \cite{Stein}, the recurrence formula for
$D(d,n,k)$ is presented. Based on Theorem \ref{pr:1},  the
recurrence relation  (\ref{eq:recbspline}) of B splines implies that
the recurrence relations of $A_{d,k}$ and $D(d,n,k)$.
\end{remark}

\section{ Proofs of Main Results}
To prove the main results, we need introduce some properties on B
splines.

\begin{lemma}\label{le:1}{\rm (\cite{cs})}

{\rm (i)}The function  $\log B_d(\cdot)$ is concave on the open
interval $(0,d)$;

{\rm (ii)}The B spline $B_d(x)$ satisfies the following two-scale
equation
 \begin{equation}\label{eq:8}
 B_d(x)=\sum_{j=0}^{d}2^{-d+1}\binom {d}{j}B_d(2x-j).
\end{equation}
\end{lemma}
We are now ready to prove our results.

\begin{proof}[Proof of Theorem \ref{pr:1}]
We firstly prove (i).
 Noting that (\ref{eq:1.1}) and (\ref{eq:unitBspline}), we have
\begin{eqnarray*}
A_{d,k}=V(T^{d}_{k})&=&d!\cdot\int_{k-1}^{k}B_d(x)dx\\
 &=& d!\cdot \int_0^1 B_d(k-t)dt = d!\cdot B_{d+1}(k).
\end{eqnarray*}
We now prove {\rm (iii)} for convenience. Let
$$Y_k=\{y\in
C^{d}:(k-1)+\frac{1}{n}\leq\sum_{i=1}^dy_{i}\leq k+\frac{1}{n}\}.$$
In view of $X_{n,k}^d=nY_k$, there holds $V(X_{n,k}^d)=n^d\cdot
V(Y_k)$. From (\ref{eq:unitBspline}), we have
$$V(X_{n,k}^d)=n^d\cdot {V}(Y_k)
=d!\cdot n^d\int_{(k-1)+\frac{1}{n}}^{k+\frac{1}{n}}\! B_d(x)dx.
$$
Then (iii) follows from the integral evaluation:
\begin{eqnarray*}
D(d,n,k)\,&=&\,V(X^d_{n,k})=
d!\cdot n^d\int_{(k-1)+\frac{1}{n}}^{k+\frac{1}{n}}B_d(x)dx\\
\,&=&\,d!\cdot n^d\int_0^1 B_{d}(k+\frac{1}{n}-t)dt=d!\cdot n^d
B_{d+1}(k+\frac{1}{n}).
\end{eqnarray*}
To prove (ii), recall the equation (cf. \cite{E})
$$X^d_{\lambda+1,k}=\lambda T^d_k +T^d_{k+1}.$$
From the argument above, one has
$$
V(X^d_{\lambda+1,k})=d!\cdot (\lambda+1)^d \int_0^1
B_{d}(k+\frac{1}{\lambda+1}-t)dt.
$$
Then, we have
\begin{eqnarray*}
\mathbf{A}_{d+1,k,d-j+1}\, &=&\,V(T^d_k,j;T^d_{k+1},d-j)\\
\,&=&\,[\lambda^j]\left( V(\lambda
T^d_k+T^d_{k+1})\right)/{\binom{d}{j}}
=[\lambda^j]\left( V(X^d_{\lambda+1,k})\right)/{\binom{d}{j}}\\
\,&=&\,[\lambda^j]\left( d!\cdot (\lambda+1)^d\int_0^1
B_{d}(k+\frac{1}{\lambda+1}-t)dt\right)/\binom{d}{j}\\
\,&=&\, [\lambda^j]\left( d!\cdot (\lambda+1)^d
B_{d+1}(k+\frac{1}{\lambda+1})\right)/\binom{d}{j}.
\end{eqnarray*}
\end{proof}
\begin{proof}[Proof of Corollary \ref{pr:2}] To prove {\rm (i)}, using (ii) in Theorem 1 and
 (\ref{eq:bspline}), for any $\lambda \geq 0 $, we have
\begin{eqnarray*}& &\mathbf{A}_{d+1,k,d-j+1}\\
\,&=&\,[\lambda^j]\left(d!\cdot (\lambda+1)^dB_{d+1}(k+\frac{1}{\lambda+1})\right)/\binom{d}{j}\\
\,&=&\,[\lambda^j]
\left(\sum_{i=0}^{d+1}\binom{d+1}{i}(-1)^i\left((\lambda+1)(k-i)+1\right)_+^{d}\right)/\binom{d}{j}\\
\,&=&\,[\lambda^j]
\left(\sum_{t=0}^d\binom{d}{t}\lambda^t\sum_{i=0}^{k}\binom{d+1}{i}(-1)^i(k-i)^t(k-i+1)^{d-t}\right)/\binom{d}{j}\\
\,&=&\,\sum_{i=0}^{k}\binom{d+1}{i}(-1)^i(k-i)^j(k-i+1)^{d-j}.
\end{eqnarray*}

To prove {\rm (ii)}, combining  (iii) in Theorem 1 with
 (\ref{eq:bspline}), we obtain that
\begin{eqnarray*}
 D(d,n,k)\,&=&\,d!\cdot n^dB_{d+1}(k+\frac{1}{n})\\
 \,&=&\,\sum_{i=0}^{k}\binom{d+1}{i}(-1)^i(n(k-i)+1)^d.
\end{eqnarray*}

\end{proof}
\begin{proof}[Proof of Corollary \ref{th:1}]
According to Theorem 1, we have
$$
 D(d,n,k)\,=\, d!\cdot n^d\cdot B_{d+1}(k+\frac{1}{n}).
$$
From Lemma \ref{le:1}, $B_{d+1}(x)$ is log-concave, which implies
that  $D(d,n,k)$, $k=0,\ldots,d$, is log-concave for each fixed $d$
and $n$.
\end{proof}
\begin{proof}[Proof of Corollary \ref{th:2}]
By Theorem 1, we have
\begin{eqnarray*}
A_{d,k} \,&=&\, d!\cdot B_{d+1}(k),\\
 D(d,n,k)\,&=&\, d!\cdot n^d\cdot B_{d+1}(k+\frac{1}{n}).
\end{eqnarray*}
Since  B-splines $B_d(x)$ satisfy the following two-scale equation
\begin{eqnarray*}
 B_d(x) =\sum_{j=0}^{d}2^{-d+1}\binom {d}{j}B_d(2x-j),\\
 \end{eqnarray*}
then we have
 \begin{equation}\label{eq:re1}
  d!\cdot B_{d+1}(k)\, =\,\sum_{j=0}^{d+1}2^{-d}\binom {d+1}{j}d!\cdot B_{d+1}(2k-j)\\
   \end{equation}
and
 \begin{equation}\label{eq:re2}
 d!\cdot n^d\cdot B_{d+1}(k+\frac{1}{2n})\,=\,\sum_{j=0}^{d+1}2^{-d}
 \binom {d+1}{j} d!\cdot n^d\cdot B_{d+1}(2k+\frac{1}{n}-j).
\end{equation}
The equations (\ref{eq:re1}) and (\ref{eq:re2}) imply that the
Eulerian numbers $A_{d,k}$ and $D(d,n,k)$ satisfy the following
two-scale equations
$$
A_{d,k}\,=\,\sum_{j=0}^{d+1} 2^{-d}\binom {d+1}{j}A_{d,{2k-j}}
$$
and
\begin{eqnarray*}
D(d,2n,k)\,=\,\sum_{j=0}^{d+1}\binom {d+1}{j}D(d,n,{2k-j}),
\end{eqnarray*}
respectively.
\end{proof}

\vspace{1cm}

{\bf Acknowledgments}  The authors are most grateful to Wenchang Chu
for discussions and comments which  improve the manuscript.

\bibliographystyle{amsplain}

\end{document}